\documentclass[12pt,reqno]{article}

\usepackage[usenames]{color}
\usepackage{graphicx}
\usepackage{amscd,amsmath,amssymb,amsthm}

\usepackage[all]{xy}
\usepackage{fullpage}
\usepackage{float}

\usepackage{hyperref,url}

\newtheorem{theorem}{Theorem}[section]
\newtheorem{lemma}[theorem]{Lemma}
\newtheorem{proposition}[theorem]{Proposition}

\theoremstyle{definition}
\newtheorem{definition}[theorem]{Definition}

\theoremstyle{remark}

\numberwithin{equation}{section}

\begin{document}

\title{\LARGE\bf Foldings and Meanders}

\author{
\bf{St\'{e}phane Legendre}\\
Team of Mathematical Eco-Evolution\\
Ecole Normale Sup\'{e}rieure\\
Paris, France\\
\texttt{legendre@ens.fr}\\}

\date{31/05/2012}

\maketitle

\begin{abstract}
We review the stamp folding problem, the number of ways to fold a strip of $n$ stamps, and the related problem of enumerating meander configurations. The study of equivalence classes of foldings and meanders under symmetries allows to characterize and enumerate folding and meander shapes. Symmetric foldings and meanders are described, and relations between folding and meandric sequences are given. Extended tables for these sequences are provided.
\end{abstract}

\section{Introduction}\label{introduction}

How many ways to fold a strip of $n$ stamps? Lucas \cite{lucas1891} stated the problem in 1891, quoting Emile Lemoine for suggesting it. Since then, though a large number of terms of the folding sequences have been computed, no closed formula has been found.

In the 1930's Sainte-Lag\"ue \cite{sainte-lague1937} lists the terms of $t(n)$ --- the number of foldings of $n$ labeled stamps --- up to $n = 10$, describing how to reduce the computation. In 1950, Touchard \cite{touchard1950} brings several interesting ideas, though not leading to a solution.

The year 1968 sees two major advances by Lunnon \cite{lunnon1968} and Koehler \cite{koehler1968}. Lunnon computes the values of $t(n)/2n$ up to $n = 24$ using a backtracking algorithm searching exhaustively for the foldings. He conjectures that the growth rate of the sequence is $3\frac{1}{2}$.

The approach of Koehler is different. He relates the enumeration of the $n$-foldings to the enumeration of patterns of chords joining $n$ points on a circle. These patterns involve the Catalan numbers. Koehler also computes the number $b(n)$ of foldings of $n$ blank stamps up to $n = 16$, using a computer program. 

The 5 foldings of 4 blank stamps made the front cover of the book of Sloane \cite{sloane1973}, the ancestor of the Online Encyclopedia of Integer Sequences \cite{OEIS} which contains several sequences concerning foldings and meanders \cite{sloane2002}. 

A fascinating appearance of the problem is the study of Phillips \cite{phillips}. A category of mazes --- simple alternating transit mazes, of which the Cretan maze \textit{circa} 604 BC is the oldest instance --- can be enumerated by the same sequence as the blank stamps foldings.

In relation with differential geometry, Arnold \cite{arnold1988} introduces the meanders, which resemble the foldings. After the pioneering work of Lando and Zvonkin \cite{lando&zvonkin1992,lando&zvonkin1993}, more recent studies mostly focus on meanders (e.g., \cite{DiFrancisco&al1996,jensen2000a,franz2002,albert&paterson2005}). In 2000, Jensen \cite{jensen2000a} could compute meandric numbers by not searching for all meanders. In this way, the terms of $r(n) = t(n)/n$, the number of stamp foldings with leaf 1 on top (a generalization of meanders, the semi-meanders), were extended up to $n = 45$. More recently, Sawada and Li \cite{sawada&li2012} describe a general algorithm to enumerate meander and folding patterns.\\

In this study, we review the various instances of foldings and meanders configurations, and describe their relations. Formulas relating folding and meander sequences are provided, and we revisit asymptotic properties. The point of view of foldings --- of which the meanders are a subclass --- is revived.

\section{Labeled stamps}\label{labeled_stamps}
The stamps are labeled $1, \ldots, n$ and folded into a stack one stamp wide and $n$ stamps tall (Fig. \ref{t4}). The horizontal segments represent the stamps, whereas the vertical segments represent the perforations.

Let $\mathcal{T}_n$ denote the set of distinct foldings of $n$ stamps. By `distinct' we mean: not counting the left-right symmetries (the stamps being stacked up as in Fig. \ref{t4}) and distinguishing the stamps (they are labeled $1,\ldots ,n$). A related problem considers blank stamps (next section).

Any folding can be represented by a permutation by listing the labels of the stamps from top to bottom (Fig. \ref{t4}). In the sequel we identify a folding with its associated element in the set of permutations $S_n$, so that a permutation pertaining to a folding is called a folding. For example, the permutation $(1324)$ is not a folding because a crossing occurs.
 
\begin{lemma}[Koehler \cite{koehler1968}]\label{condition_folding}
An $n$-permutation $p$ is a folding if and only if the circular order
\begin{displaymath}
p(i) < p(j) < p(i+1) < p(j+1)
\end{displaymath}
does not occur when $i$ and $j$ are either both odd or both even. By circular order it is meant any circular permutation of the inequalities above.
\end{lemma}
\begin{proof}
The lemma relies on the following observation, which is shown by induction: in the representation of a stamp folding (Fig. \ref{t4}), the vertical segments on the left represent the perforations between the stamps labeled $2i$ and $2i+1$, whereas the vertical segments on the right represent the perforations between the stamps labeled $2i-1$ and $2i$.
\end{proof}

\begin{proposition}[Sainte-Lag\"ue \cite{sainte-lague1937}]\label{t_n}
The set $\mathcal{T}_n$ of foldings of $n$ labeled stamps can be partionned into $r(n)$ disjoint orbits of size $n$ under the action of the circular permutation $C = (23 \cdots n1)$. The set $\mathcal{R}_n$, the $n$-foldings $(1 \cdots)$ with leaf 1 on top, is a set of representatives of the orbits of $C$ in $\mathcal{T}_n$. Hence, the number of elements of $\mathcal{T}_n$  is
\begin{equation}\label{tr_n}
t(n) = nr(n).
\end{equation} 
\end{proposition}
\begin{proof}
Each row of Fig. \ref{t4} gives a visual argument that if $f$ is a folding then so is its image $f^C$ under the circular permutation $C$ (a more rigorous argument is provided by Lemma \ref{condition_folding}). As $C^{n} = 1$, each orbit contains $n$ elements. Moreover, each orbit contains exactly one folding with leaf 1 on top, and these foldings are distinct.
\end{proof}

\begin{figure}[ht]
\begin{center}
\includegraphics[scale=0.6]{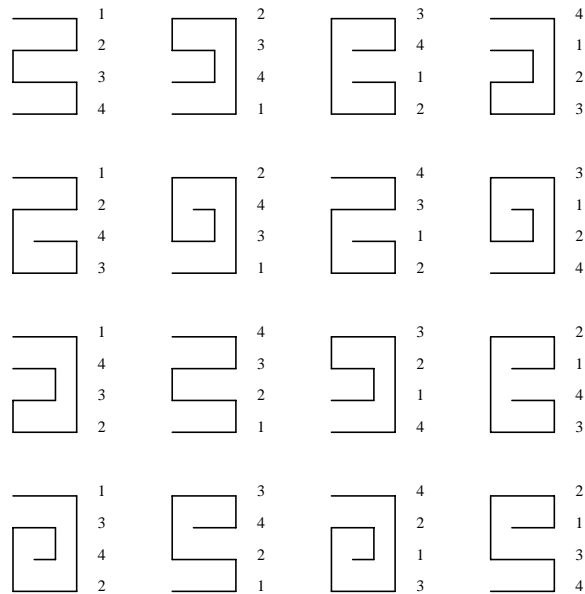}
\caption{The 16 foldings of 4 labeled stamps constituting the set $\mathcal{T}_4$. The 4 foldings in each row form an orbit under the circular permutation: the upper leaf is successively `rotated' to the bottom. The foldings in the first column with leaf 1 on top are representatives of the orbits, constituting the set $\mathcal{R}_4$.}
\label{t4}
\end{center}
\end{figure}

According to Proposition \ref{t_n}, $\mathcal{T}_n$ is stable under the circular permutation. Lemma \ref{condition_folding} shows that this is also the case of its complement $S_n - \mathcal{T}_n$: if $p$ is a permutation that is not a folding then $p^C$ is not a folding either. 

The set $\mathcal{R}_n$ can be constructed inductively using a tree $n$ levels deep (Fig. \ref{tree_r}). The root of the tree (level 1) is the folding constituted of a single leaf. Level $n$ of the tree describes $\mathcal{R}_n$ and comprises $r(n)$ nodes. To each folding (node) $f \in \mathcal{R}_n$ at level $n$, a leaf labeled $n+1$ is appended to leaf $n$. The new leaf is inserted in any possible position with the constraint that leaf 1 must stay on top, leading to all descendant nodes of $f$. 

As apparent in Fig. \ref{tree_r}, below level 2, the tree is made of two topologically identical subtrees. The reason is that if $(1 ab \cdots yz)$ is a folding then so is the reversed $(1 zy \cdots ba)$. Each folding $(1 \cdots 2 \cdots 3 \cdots)$ in the left subtree is uniquely associated with a folding $(1 \cdots 3 \cdots 2 \cdots)$ in the right subtree. Hence, for $n > 2$, $r(n)$ is even.

\begin{figure}[ht]
\begin{center}
\includegraphics[scale=0.6]{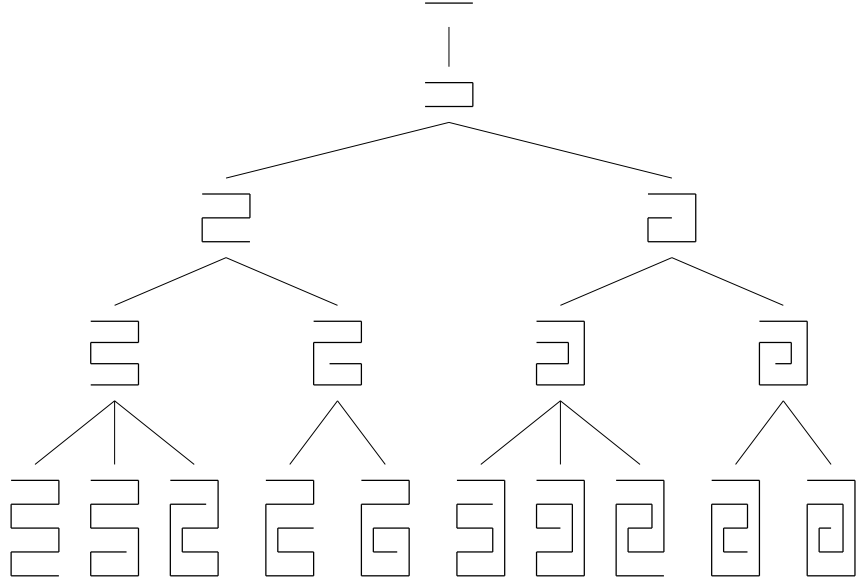}
\caption{The tree describing the sets $\mathcal{R}_n$ for $n = 1,\ldots ,5$. At a given level of the tree, each folding in the left subtree can be paired uniquely with a folding in the right subtree by keeping leaf 1 in place and `rotating' the leaves around leaf 2.}
\label{tree_r}
\end{center}
\end{figure}

The algorithm of Lunnon \cite{lunnon1968} lists the foldings $(1 \cdots 2 \cdots 3 \cdots)$ using a depth first search. The algorithm explicitly looks for all positions where to insert the new leaf, keeping leaf 1 on top. 

\section{Blank stamps}\label{blank_stamps}
When the stamps are not labeled (blank stamps), only the distinct shapes of the foldings are considered. While the foldings of $\mathcal{T}_n$ may share identical shapes (but with different corresponding permutations, Fig. \ref{t4}), the set $\mathcal{R}_n$ is constituted of distinct shapes, but with the restriction that leaf 1 is on top.

Two maps on the set of permutations $S_n$ will be important for the study of foldings and meander shapes, the reverse map $r$ and the complement map $c$. For  $p \in S_{n}$, $r$ reverses the entries of $p$: $p^r(i) = p(n+1-i)$ for $i=1, \ldots ,n$. This corresponds to the top-bottom symmetry of a folding. Hence, if $f$ is a folding, $f^r$ is a folding with the same shape as $f$, but top-bottom reflected. The complement map $c$ is defined by $p^c(i) = n+1 - p(i)$ for $i=1, \ldots ,n$, and corresponds to label reversal. If $f$ is a folding, $f^c$ is a folding with the same shape as $f$, but traversed in the reverse order $n, \ldots, 1$ of the labels.

The reverse and complement maps $r$ and $c$ verify $r^2=c^2=1$, $rc=cr$, so that they generate a group of 4 elements operating on $S_n$, $G = \{1,r,c,rc\}$. If $x$ is a map on $S_n$ and $f^x$ has the same shape as $f$, it can be shown that $x \in G$. Thus, the folding shapes are the equivalence classes of $\mathcal{T}_n$ under the action of $G$. We adopt the following definition (Egge \cite{egge2007}):
\begin{definition}\label{symmetry}
A permutation is \textit{symmetric} if it is invariant under the reverse-complement map $rc$.
\end{definition} 
Thus, $p$ symmetric is equivalent to $p(i) = n+1-p(n+1-i)$, $i=1, \ldots ,n$. Let us check that a folding $f$ that is symmetric as a permutation has a (top-bottom) symmetric shape. Indeed, if $f^{rc}=f$, we have $f^c=f^r$. This means that the reversed shape $f^r$ is identical to the (not reversed) shape $f^c$ obtained by traversing the folding in the reverse order of the labels. Since $f^c$ has the same shape as $f$, the shape of $f$ is symmetric. For $n=2p+1$ odd, a symmetric folding $f$ has leaf $p+1$ in central position, $f(p+1) = p+1$.

\begin{proposition}[Koehler \cite{koehler1968}]\label{blank}
Let $b(n)$ denote the number of foldings of $n$ blank stamps, $z(n)$ the number of symmetric foldings of $n$ labeled stamps. Then, for $n>1$,
\begin{equation}\label{b_n}
b(n) = \frac{t(n)+z(n)}{4}.
\end{equation}
Moreover, $z(n)$ is even, and for $n=2p$ even,
\begin{equation}\label{z_2p}
z(2p) = 2r(p+1).
\end{equation} 
\end{proposition}
\begin{proof}
We have seen that for $f \in \mathcal{T}_n$, $f^x$ is a folding with the same shape as $f$ if and only if $x \in G$. In $\mathcal{T}_n$, there are either 4 distinct foldings with the same shape ($f$, $f^r$, $f^c$, and $f^{rc}$) or only 2 distinct foldings with the same shape ($f$ and $f^r=f^c$) when $f$ is symmetric (Fig. \ref{fold_meand_rc}). Consequently, the number of distinct shapes is 
\begin{displaymath}
b(n) = \frac{t(n)-z(n)}{4} + \frac{z(n)}{2} = \frac{t(n)+z(n)}{4},
\end{displaymath}
and we also know that $z(n)$ is even.

Relation (\ref{z_2p}) comes from the fact that for $n=2p$, any symmetric $2p$-folding is uniquely associated with a $(p+1)$-folding of $\mathcal{R}_{p+1}$ (Fig. \ref{manip_fold}).
\end{proof}

\begin{figure}[ht]
\begin{center}
\includegraphics[scale=0.6]{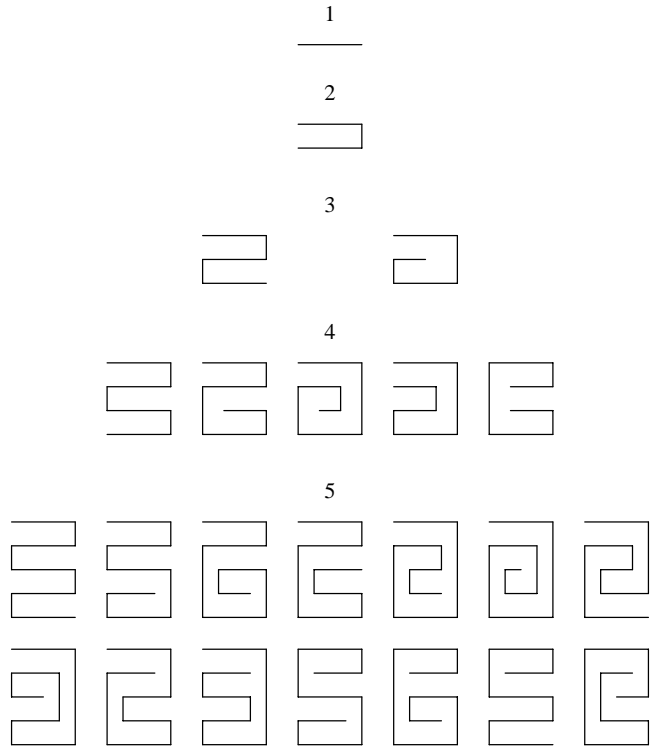}
\caption{The foldings of $n$ blank stamps for $n = 1, \ldots, 5$.}
\label{b_1_5}
\end{center}
\end{figure}

From the proof of Proposition \ref{blank}, we understand that $z(n)/2$ is the number of symmetric foldings of $n$ blank stamps. For $n=2p$ even, the shapes of the symmetric foldings are 2-copies of those of $\mathcal{R}_{p+1}$, in number
\begin{displaymath}
\frac{z(2p)}{2} = r(p+1).
\end{displaymath}

\begin{figure}[t]
\begin{center}
\includegraphics[scale=0.6]{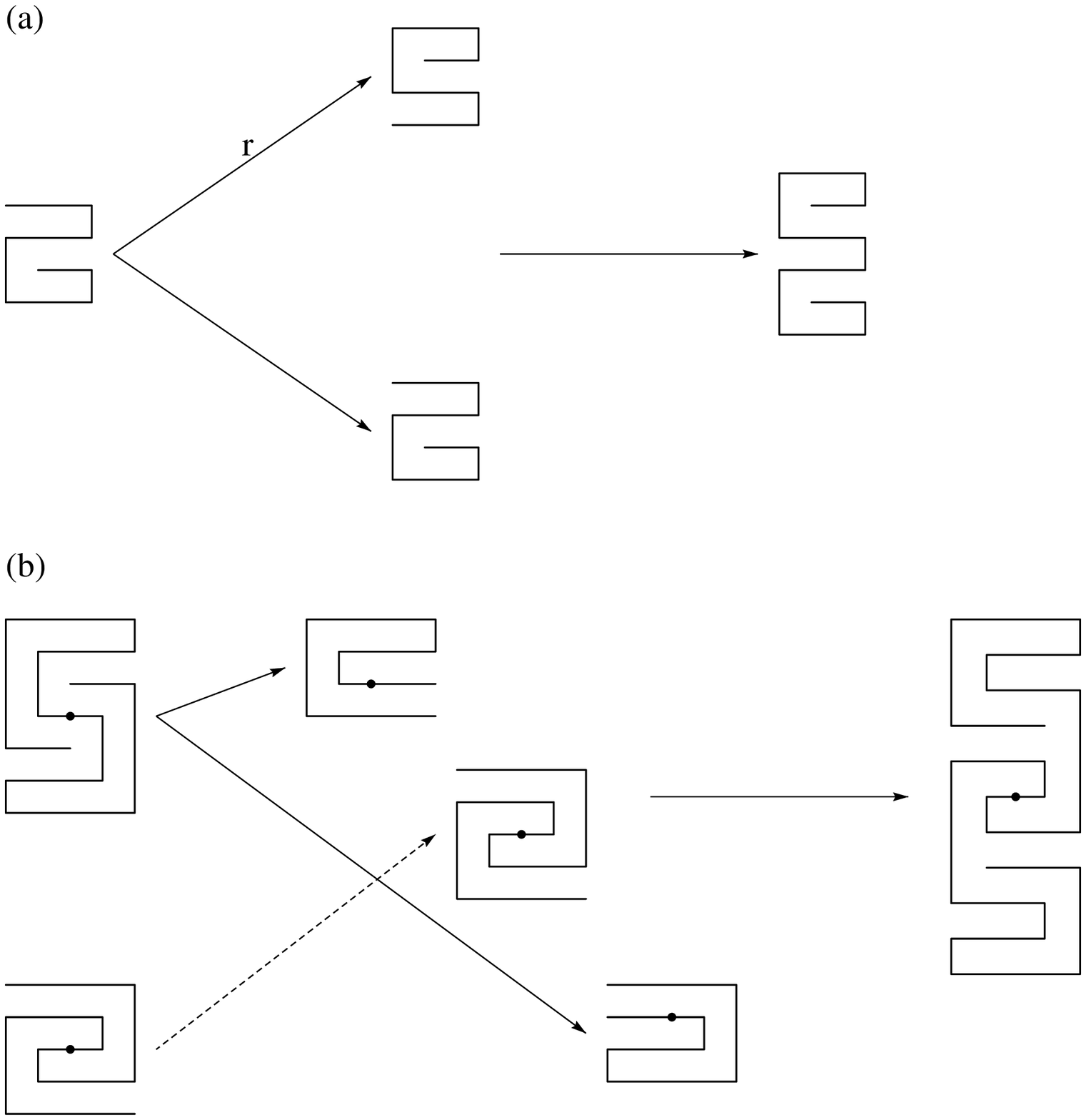}
\caption{(a) For $n=2p$ even, any symmetric folding of $n$ stamps can be associated uniquely with a folding $(1 \cdots )$ of $p+1$ labeled stamps by concatening 2 copies and discarding 2 stamps. (b) A symmetric $(2p+1)$-folding and a symmetric $(2q+1)$-folding produce a symmetric $(2n+1)$-folding with $n=p+q$.}
\label{manip_fold}
\end{center}
\end{figure}

For $n=2p+1$ odd, the shapes of the symmetric foldings are more intricate. The number of symmetric foldings of $2p+1$ blank stamps is
\begin{displaymath}
k(p) = \frac{z(2p+1)}{2}.
\end{displaymath}

The corresponding set can be constructed inductively using a tree (Fig. \ref{tree_sym_odd}), as for the set $\mathcal{R}_n$. At each stage of the process, 2 leaves are appended to the free ends of the folding. The new leaves are inserted symmetrically in any possible position.

\begin{figure}[ht]
\begin{center}
\includegraphics[scale=0.6]{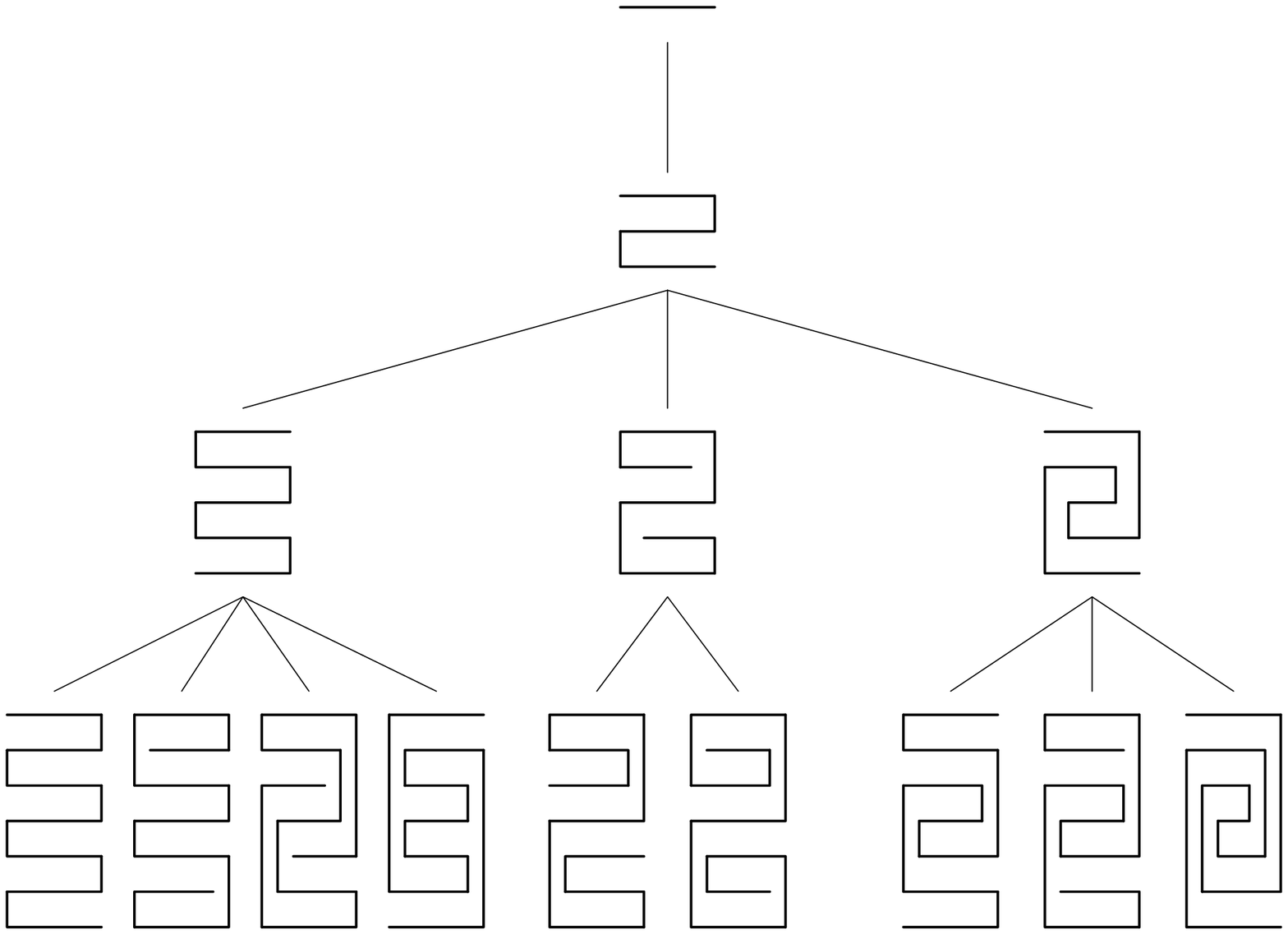}
\caption{The tree describing the symmetric foldings of $2p+1$ stamps, $p = 0, \ldots, 3$.}
\label{tree_sym_odd}
\end{center}
\end{figure}

\section{Leaf in - leaf out}\label{in-out}
We have seen that $\mathcal{T}_n$ is stable under the action of the circular permutation $C$. This suggests to consider the set $S_n$ of permutations on $n$ symbols generated by cyclic shift, containing $\mathcal{T}_n$ as a subset. The cyclic shift is a map generating $S_i$ from $S_{i-1}$ by adding a symbol to the left of a permutation of $S_{i-1}$ and circularly permuting the symbols \cite{knuth2005} \cite{legendre2011}.
\begin{definition}\label{def_in_out}
In $\mathcal{T}_n$ we distinguish the foldings with leaf $n$ \textit{out}, where the leaf points directly to the exterior (e.g., the second folding in Fig. \ref{t4}), forming  $\mathcal{T}^o_n$, and the foldings with leaf $n$ \textit{in}, where the leaf is in inner position (e.g., the third folding in Fig. \ref{t4}), forming  $\mathcal{T}^i_n$.
\end{definition} 
\begin{proposition}\label{To_Ti}
Let $t^o(n)$ denote the number of foldings of $\mathcal{T}_n$ with leaf $n$ out, $t^i(n)$ the number of foldings with leaf $n$ in, then
\begin{equation}\label{toi}
t(n) = t^o(n) + t^i(n) = nt^o(n-1). 
\end{equation}
\end{proposition}
\begin{proof}
To a folding $f \in \mathcal{T}^o_{n-1}$, since leaf $n-1$ points to the exterior, a leaf $n$ can be appended to this leaf at the bottom of the folding, producing an $n$-folding $f'$. Circular permutations of $f'$ are also foldings, so that from $f$ the cyclic shift generates $n$ foldings. By contrast, if $f \in \mathcal{T}^i_{n-1}$, the  permutations generated from $f$ by the cyclic shift are not foldings, because appending leaf $n$ at the bottom of the folding produces a crossing. As the cyclic shift generates all permutations of $S_n$ from $S_{n-1}$, we obtain all foldings of $\mathcal{T}_n$ from $\mathcal{T}^o_{n-1}$. This gives the second equality in (\ref{toi}). Moreover, from (\ref{tr_n}),
\begin{displaymath}
t^o(n) = r(n+1).   
\end{displaymath}
\end{proof}
From (\ref{toi}), we have
\begin{displaymath}
t(n) = n\left(t(n-1) - t^i(n-1)\right). 
\end{displaymath}
Expanding the recurrence, and using $t(1)=1$, we obtain
\begin{displaymath}
t(n) = n! - \sum_{j=0}^{n-2} n(n-1) \cdots (n-j) t^i(n-j-1).   
\end{displaymath}
The sum can be interpreted as the number of $n$-permutations that are not foldings. Indeed, from any folding in $\mathcal{T}^i_{n-j-1} \subset S_{n-j-1}$, successive applications of the cyclic shift generate $n(n-1) \cdots (n-j)$ permutations of $S_n$ that are not foldings. Hence, all permutations of $S_n$ that are not foldings are generated by cyclic shift from the foldings of $S_3, \ldots, S_{n-1}$ with the last leaf in (noting that $t^i(1) = t^i(2) = 0$, Table \ref{table_io}). The formula can be rewritten 
\begin{equation}\label{propfold}
\frac{t(n)}{n!} = 1 - \sum_{k=3}^{n-1} \frac{t^i(k)}{k!}.   
\end{equation}

Among the elements of $\mathcal{T}^o_{n}$, we have those with leaf 1 out, forming $\mathcal{T}^{oo}_{n}$, and those with leaf 1 in, forming $\mathcal{T}^{io}_{n}$. In the same way, among the elements of $\mathcal{T}^i_{n}$, we have those with leaf 1 out, forming $\mathcal{T}^{oi}_{n}$, and those with leaf 1 in, forming $\mathcal{T}^{ii}_{n}$. Moreover, the sets $\mathcal{T}^{io}_{n}$ and $\mathcal{T}^{oi}_{n}$ are in bijective correspondance under the complement map $c$. With obvious notations,
\begin{equation}\label{tooii}
t(n) = t^{oo}(n) + 2t^{io}(n) + t^{ii}(n).
\end{equation} 

\begin{table}[ht]
\footnotesize
\begin{center}
\begin{tabular}{| l | r r r r r r r r r r r r r r |}
\hline
$n$             &1 &2 &3 &4  &5  &6  &7   &8   &9    &10   &11    &12     &13     &14\\
\hline
$t^o$           &1 &2 &4 &10 &24 &66 &174 &504 &1406 &4210 &12198 &37378  &111278 &346846
\\\hline
$t^{oo}$        &1 &2 &2 &6  &8  &28 &42  &162 &262  &1076 &1828  &7852   &13820 &61388\\
$t^{io}=t^{oi}$ &0 &0 &2 &4  &16 &38 &132 &342 &1144 &3134 &10370 &29526  &97458 &285458\\
$t^{ii}$        &0 &0 &0 &2  &10 &40 &156 &546 &1986 &6716 &23742 &79472  &277178 &925588\\\hline
$t^i$           &0 &0 &2 &6  &26 &78 &288 &888 &3130 &9850 &34112 &108998 &374636 &1211046\\
\hline
\end{tabular}
\end{center}
\caption{The number of foldings of $n$ labeled stamps with leaf $1$, $n$ in or  out.}
\label{table_io}
\end{table}

\section{Meanders}\label{meanders}
A meander of order $n$ is a simple oriented curve in the plane transversally intersecting a line at $n$ points. The line can be considered as a road traveling from W to E, and the meander as a river flowing from SW to E (NE for odd $n$, SE for even $n$), and going through $n$ bridges across the road (Fig. \ref{m_1_5}). The number of meanders of order $n$ is denoted $m(n)$. 

\begin{figure}[ht]
\begin{center}
\includegraphics[scale=0.75]{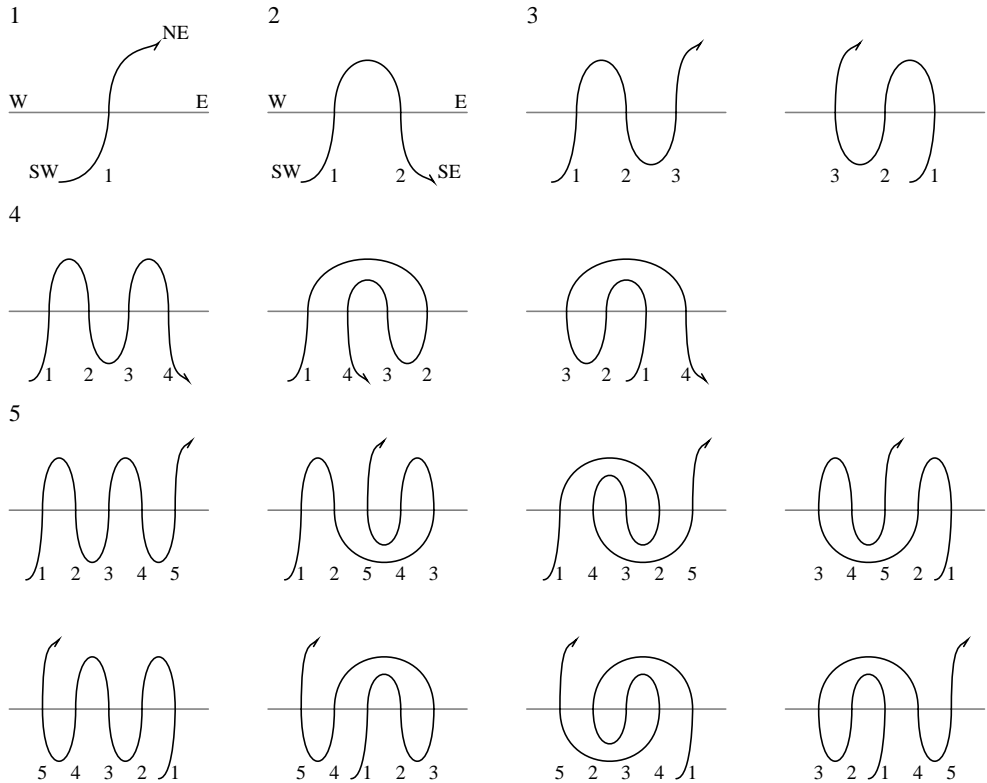}
\caption{Meanders through $n = 1, \ldots, 5$ bridges.}
\label{m_1_5}
\end{center}
\end{figure}                  

Joining the ends of a meander of odd order $2n-1$ creates a closed meander of order $2n$. This operation establishes a one-to-one correspondance between meanders of order $2n-1$ and closed meanders of order $2n$. Hence, the number $M(n)$ of closed meanders through $2n$ bridges verifies
\begin{displaymath}
M(n) = m(2n-1).
\end{displaymath} 

Any meander can be viewed as a folding but the converse is not true. The first folding in the second row of Fig. \ref{t4} is not a meander because leaf $n$ points inward. A \textit{semi-meander} is a generalization of a meander where one of the end points is allowed to be wound inside the river. Semi-meanders are in one-to-one correspondance with foldings $(1 \cdots)$ of $n+1$ labeled stamps, in number $r(n+1)$ \cite{DiFrancisco&al1996}. If both ends of the river are allowed to wound up inside itself, then we obtain a stamp folding \cite{sawada&li2012}. In fact, the next proposition shows that the set of meanders of order $n$ is almost identical to the set $\mathcal{T}^{oo}_n$ of stamp foldings whose both ends are out. The slight difference comes from the fact that a meander is oriented while a folding is not.

Like for foldings, permutations are associated with meanders (by listing the bridges). A permutation representing a meander is also called a meander.

\begin{proposition}\label{moo}
The number of meanders is
\begin{eqnarray}
\label{oo_odd}   m(2p-1) &=& t^{oo}(2p-1),\\
\label{oo_even}  m(2p)   &=& \frac{t^{oo}(2p)}{2},
\end{eqnarray}
where $t^{oo}(n)$ is the number of $n$-foldings with both ends out. Moreover,
\begin{equation}\label{m_ro}   
m(n) = r^{o}(n+1),
\end{equation}
where $r^{o}(n+1)$ is the number of $(n+1)$-foldings $(1 \cdots)$ with leaf $n+1$ out.
\end{proposition}
\begin{proof}
For $n = 2p-1$ odd, there is a bijective correspondance between meanders and foldings with both ends out. For $n = 2p$ even, the image $g^r$ of a meander $g$ under the reverse map $r$ is \textit{not} a meander, because the river does not travel SW-SE in this case, but in the reverse direction (Fig. \ref{fold_meand_rc}). However, $g^r$ is a folding. Hence 2 foldings of $\mathcal{T}^{oo}_n$, $g$ and $g^r$, correspond to a single meander, $g$. 

To show (\ref{m_ro}), let us consider an $(n+1)$-folding $(1 \cdots)$ that has leaf 1 on top and leaf $n+1$ out. If leaf 1 is deleted, leaf 2 is now out, so that both ends are out. The resulting $n$-folding is an $n$-meander.
\end{proof}

\begin{figure}[ht]
\begin{center}
\includegraphics[scale=0.6]{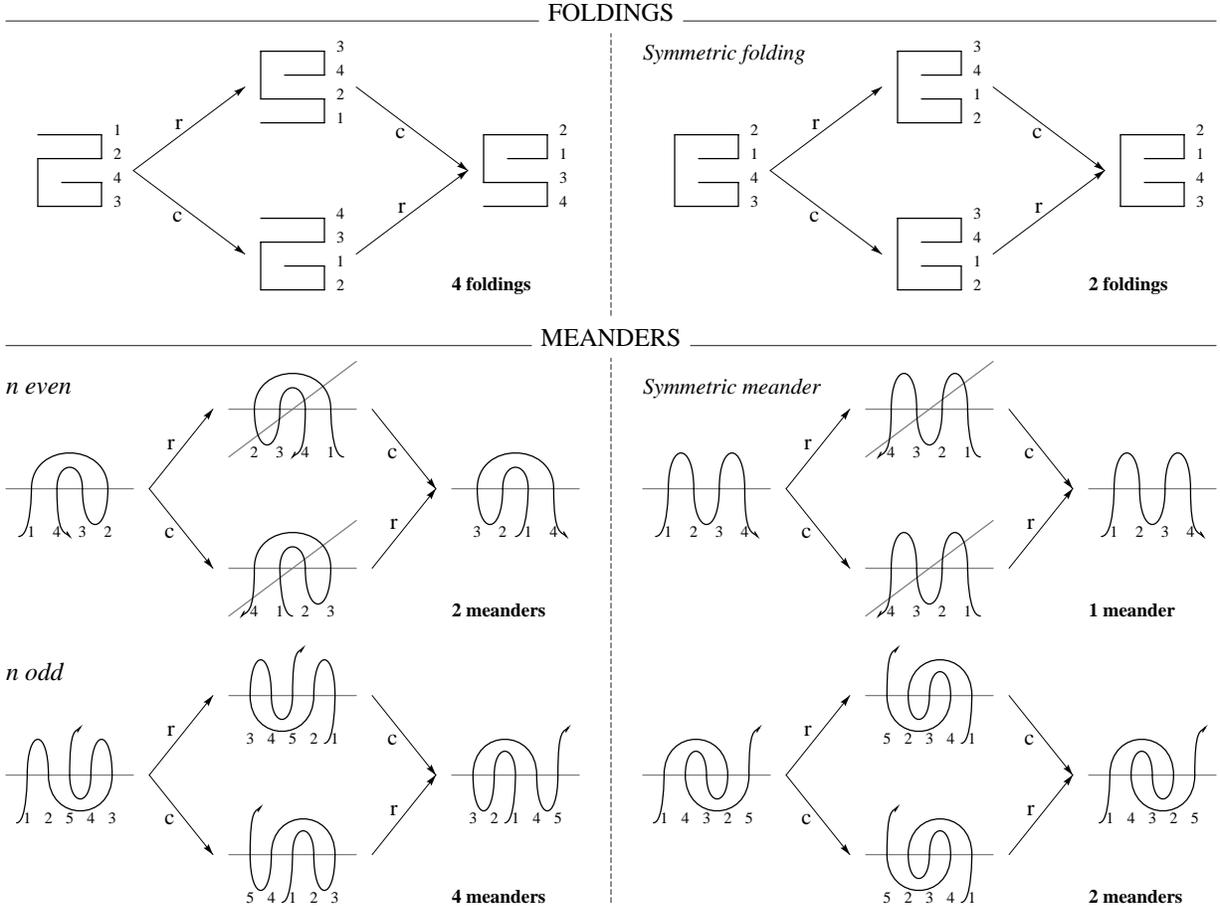}
\caption{The action of the reverse map $r$ and the complement map $c$ on foldings and meanders. In each case, a single shape is obtained.}
\label{fold_meand_rc}
\end{center}
\end{figure}   

\section{Meander shapes}\label{meander_shapes}
In the same way that blank stamps foldings describe the various shapes of labeled stamps foldings, the \textit{meander shapes} are defined as the equivalence classes of meanders under the reverse and complement maps $r$ and $c$. As seen in the previous section, one has to distinguish $n$ odd and $n$ even. Fig. \ref{fold_meand_rc} illustrates the following points. For $n$ even, two meanders are equivalent if one is the image of the other under the reverse-complement map $rc$, a reflection in the Y axis (map $r$) together with a relabeling of the bridges reversing the orientation (map $c$). For $n$ odd, two meanders are equivalent if one is the image of the other under the map $r$, a reflection in the Y axis, or the complement map $c$, an orientation reversal along the X axis. As for foldings, a meander is \textit{symmetric} if it is invariant under the reverse-complement map $rc$.
\begin{proposition}\label{prop_meander_shapes}
The number of meander shapes is
\begin{eqnarray}
\label{a_even}  a(2p) &=& \frac{m(2p) + m(p)}{2},\\
\label{a_odd} a(2p+1) &=& \frac{m(2p+1) + 2k^o(p)}{4},
\end{eqnarray}
where $k^o(p)$ is the number of symmetric $(2p+1)$-folding shapes with both ends out. Moreover,
\begin{equation}\label{a_boo}  
a(n) = b^{oo}(n),
\end{equation}
where $b^{oo}(n)$ is the number of $n$-folding shapes with both ends out.
\end{proposition}
\begin{proof}
For $n=2p$ even, if $g$ is a meander, neither $g^r$ nor $g^c$ is a meander (Fig. \ref{fold_meand_rc}). However, $g^{rc}$ is a meander, which is identical to $g$ if and only if $g$ is symmetric. Hence, in each equivalence class, there are either 2 elements, or 1 element if $g$ is symmetric. Let $q(n)$ denote the number of symmetric $n$-meanders. Then,
\begin{displaymath}
a(2p) = \frac{m(2p) - q(2p)}{2} + q(2p) = \frac{m(2p) + q(2p)}{2}.
\end{displaymath}
Symmetric $2p$-meanders are the concatenation of two $p$-meanders, so that $q(2p)=m(p)$, giving (\ref{a_even}). For $n=2p+1$ odd, there are either 4 or 2 elements in each equivalence class, the later case occuring for symmetric meanders. We have
\begin{displaymath}
a(2p+1) = \frac{m(2p+1) - q(2p+1)}{4} + \frac{q(2p+1)}{2} = \frac{m(2p+1) + q(2p+1)}{4}.
\end{displaymath}
Each symmetric folding $f$ of $2p+1$ blank stamps that has both ends out defines exactly two distinct $(2p+1)$-meanders, $f$ and $f^r$ (hence a single meander shape). Thus $q(2p+1)=2k^o(p)$, giving (\ref{a_odd}).

Using Propositions \ref{blank} and \ref{moo} together with (\ref{a_even}) and (\ref{a_odd}), formula (\ref{a_boo}) comes from the following relations:
\begin{displaymath}
b^{oo}(2p+1) = \frac{t^{oo}(2p+1) + 2k^o(p)}{4} = \frac{m(2p+1) + 2k^o(p)}{4} = a(2p+1),
\end{displaymath}
\begin{displaymath}
b^{oo}(2p) = \frac{t^{oo}(2p) + 2r^o(p+1)}{4} = \frac{2m(2p) + 2m(p)}{4} = a(2p).
\end{displaymath}
\end{proof}

\section{Asymptotics}\label{asymptotics}

By Fekete's lemma \cite{vanlint&wilson1992}, if a sequence $\textbf{x}$ verifies
\begin{equation}\label{fekete}
x(p)x(q) \le x(p+q) \mbox{ for all } p,q,
\end{equation}
then the growth rate of the sequence, $\lambda = \lim_{n \rightarrow \infty} x(n)^{\frac{1}{n}}$, exists. 

The sequences $\textbf{r}$ and $\textbf{k}$ verify (\ref{fekete}) and hence have a growth rate. The reason is that any $p$-folding can be concatenated with any $q$-folding to produce a $(p+q)$-folding. This is shown for the symmetric foldings in Fig. \ref{manip_fold}.

It is generally believed \cite{lando&zvonkin1993,DiFrancisco&al1996,bacher1999,jensen2000b}, that folding and meander sequences verify asymptotic relations of the form
\begin{equation}\label{asymptotic}
x(n) \sim K n^{\alpha}\lambda^{n},
\end{equation}
with the exponent $\alpha$, and $K$ a constant. In this case, we have $\frac{x(n+1)}{x(n)} \sim \left(1+\frac{1}{n}\right)^{\alpha} \lambda \sim \lambda$. 

By (\ref{tr_n}), sequences $\textbf{t}$ and $\textbf{r}$ have the same growth rate $\lambda$ (with different exponents $\alpha$). Sequence $\textbf{b}$ also has the growth rate $\lambda$ by Proposition \ref{blank}, and so is most probably the case for the sequence $\textbf{k}$ of symmetric folding shapes (Fig. \ref{growthrate_rbk}). From Table \ref{table_stamps} we get the following estimates:
\begin{displaymath}
 \frac{t(45)}{t(44)} = 3.4385110353517, \quad \frac{b(45)}{b(44)} = 3.4385110353636, \quad \frac{k(25)}{k(24)} = 3.426089.
\end{displaymath}
It is believed that the meandric sequences of odd and even terms of $\textbf{m}$ have the growth rate $\lambda^2$ \cite{DiFrancisco&al1996,jensen2000b}. These sequences involve different constants $K_1$ and $K_2$. In Fig. \ref{growthrate_mako}, approximations of $\frac{K_1}{K_2}\lambda$ and $\frac{K_2}{K_1}\lambda$ are displayed alternatively, and this explains the non decaying oscillations. The same pattern holds for the sequence $\textbf{a}$ of meander shapes by Proposition \ref{prop_meander_shapes}. The sequence $\textbf{k}^o$ of symmetric meander shapes also probably has the growth rate $\lambda$ (Fig. \ref{growthrate_mako}). So it seems that all folding and meander sequences studied here have the same growth rate, conjectured to be $\lambda = 3\frac{1}{2}$ \cite{lunnon1968}. This value is compatible with theoretical bounds \cite{albert&paterson2005} as well as with numerical estimations \cite{jensen2000b}.

\begin{proposition}\label{propii}
For large $n$, most foldings have both ends in. The probability for an $n$-folding to be an $n$-meander is asymptotically 0.
\end{proposition}
\begin{proof}
From (\ref{toi}), and using the growth rate $\lambda$, we have $(n+1)t^o(n) = t(n+1) \sim \lambda t(n)$, so that
\begin{displaymath}
\frac{t^o(n)}{t(n)} \sim \frac{\lambda}{n+1}.
\end{displaymath}
This means that the proportion of foldings of $\mathcal{T}_n$ with leaf $n$ out is asymptotically 0. This implies $\frac{t^{oo}(n)}{t(n)} \rightarrow 0$ and  $\frac{t^{io}(n)}{t(n)} \rightarrow 0$. By (\ref{tooii}), the proportion of foldings with both ends in verifies $\frac{t^{ii}(n)}{t(n)} \rightarrow 1$. The second assertion comes from Proposition \ref{moo}.
\end{proof} 

If $\alpha$ is the exponent of $\textbf{t}$, the exponent of $\textbf{r}$ is $\alpha-1$ by (\ref{tr_n}). Fig. \ref{fitrm} shows a very good fit 
\begin{displaymath}
r(n) \sim K n \ln(n) m(n),
\end{displaymath}
where the constant $K$ depends on the parity of $n$. This suggests that the exponent of $\textbf{m}$ is $\alpha-2$. The exponent of the meander shapes sequence $\textbf{a}$ should also be $\alpha-2$ by Proposition \ref{prop_meander_shapes}.\\

The quasi-exponential growth of $t(n)$ implies, by (\ref{propfold}), that the proportion of permutations of $S_n$ that are foldings $\frac{t(n)}{n!} \rightarrow 0$ (since $n!$ grows like $\left(\frac{n}{e}\right)^{n+\frac{1}{2}}$). Nevertheless, as noted at the end of section \ref{in-out}, any permutation is either a folding or is generated by cyclic shift from a lower order permutation that is a folding. From (\ref{propfold}), we deduce
\begin{displaymath} 
\sum_{k=3}^{\infty} \frac{t^i(k)}{k!} = 1.
\end{displaymath}
Also from (\ref{propfold}), we obtain an expression for the number of stamp foldings of order $n$ in terms of foldings of order $\ge n$:
\begin{displaymath}
\frac{t(n)}{n!} = \sum_{k=n}^{\infty} \frac{t^i(k)}{k!}.
\end{displaymath}

\appendix
\nonumber \section{Computation of the Tables}
The values of Jensen \cite{jensen2000a} have been used for the folding sequence $\textbf{r}$ and the meander sequence $\textbf{m}$. The terms of the sequence $\textbf{k}$ of symmetric folding shapes were computed using a variant of Lunnon's algorithm \cite{lunnon1968}. Among these shapes, those with leaf $n$ out were determined, giving the sequence $\textbf{k}^o$. Then the terms of the sequences $\textbf{b}$ and $\textbf{a}$ of folding and meander shapes were computed using Propositions \ref{blank} and \ref{prop_meander_shapes} respectively.

\begin{figure}[ht]
\begin{center}
\includegraphics[scale=0.9]{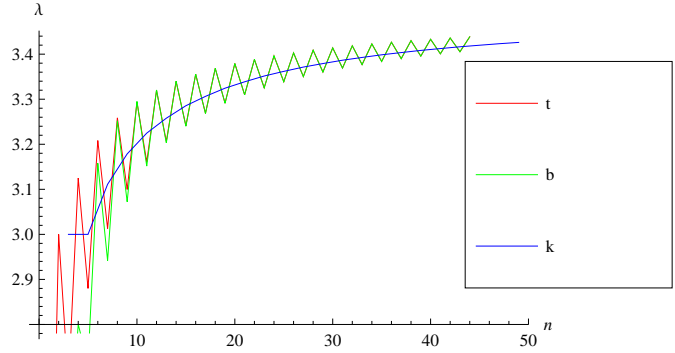}
\caption{One-step growth rates of the sequences $t(n)$, $b(n)$, and $k(p)$ ($k(p+1)/k(p)$ is plot at $n=2p+1$).}
\label{growthrate_rbk}
\end{center}
\end{figure}

\begin{figure}[ht]
\begin{center}
\includegraphics[scale=0.9]{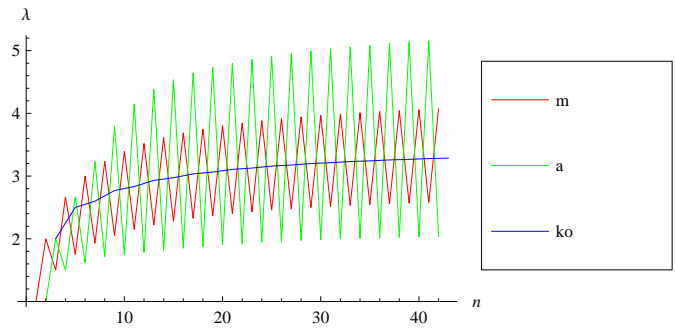}
\caption{One-step growth rates of the sequences $m(n)$, $a(n)$, and $k^o(p)$ ($k^o(p+1)/k^o(p)$ is plot at $n=2p+1$).}
\label{growthrate_mako}
\end{center}
\end{figure}

\begin{figure}[ht]
\begin{center}
\includegraphics[scale=0.8]{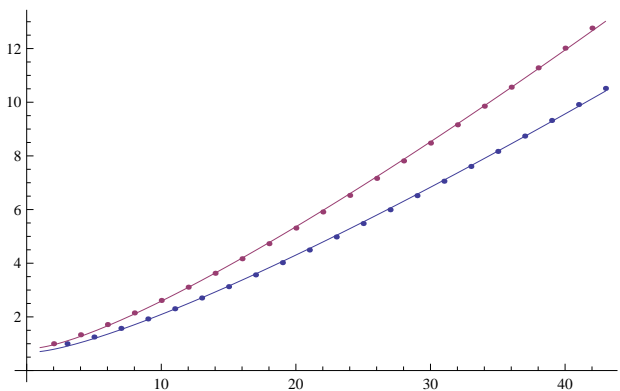}
\caption{The ratios $r(n)/m(n)$ (dots) are well fitted by $n\ln(n)$ (lines) 
for $n$ odd and even (procedure \textit{LinearModelFit} of Mathematica \cite{mathematica2010}).}
\label{fitrm}
\end{center}
\end{figure}

\begin{table}[ht]
\footnotesize
\begin{center}
\begin{tabular}{| l l l l |}
\hline
$n$ &$r(n)$ &$b(n)$ &$k(p)$\\
\hline
1	&1                     &1								&1\\
2	&1                     &1								&\\
3	&2                     &2								&1\\
4	&4                     &5								&\\
5	&10                    &14								&3\\
6	&24                    &38								&\\
7	&66                    &120							&9\\
8	&174                   &353							&\\
9	&504                   &1148							&28\\
10	&1406                  &3527							&\\
11	&4210                  &11622							&89\\
12	&12198                 &36627							&\\
13	&37378                 &121622						&287\\
14	&111278                &389560						&\\
15	&346846                &1301140						&935\\
16	&1053874               &4215748						&\\
17	&3328188               &14146335						&3072\\
18 &10274466              &46235800						&\\
19	&32786630              &155741571					&10157\\
20	&102511418             &512559195					&\\
21	&329903058             &1732007938					&33767\\
22	&1042277722            &5732533570					&\\
23	&3377919260            &19423092113					&112736\\
24	&10765024432           &64590165281					&\\
25	&35095839848           &219349187968				&377836\\
26	&112670468128          &732358098471				&\\
27	&369192702554          &2492051377341				&1270203\\
28	&1192724674590         &8349072895553				&\\
29	&3925446804750         &28459491475593				&4282311\\
30	&12750985286162        &95632390173152				&\\
31	&42126805350798        &326482748703999			&14470629\\
32	&137494070309894       &1099952564143246			&\\
33	&455792943581400       &3760291809049416			&49005732\\
34	&1493892615824866      &12698087239648594			&\\
35	&4967158911871358      &43462640562005209			&166261653\\
36	&16341143303881194     &147070289751324061		&\\
37	&54480174340453578     &503941612931723170		&565055147\\
38	&179830726231355326    &1708391899249131306		&\\
39	&600994488311709056    &5859696262000756532		&1923186472\\
40	&1989761816656666392   &19897618166731615449		&\\
41	&6664356253639465480   &68309651603081955628		&6554868916\\
42	&22124273546267785420  &232304872236332885771	&\\
43	&74248957195109578520  &798176289858611935664	&22367933148\\
44	&247100408917982623532 &2718104498099497818482	&\\
45	&830776205506531894760 &9346232311986692725748	&76417819396\\
\hline
47 &                      &								&261335128098\\
49 &                      &                        &894597454360\\
51	&                      &							   &3064970675173\\
\hline
\end{tabular}
\end{center}
\caption{$r(n)$: number of foldings of $n$ labeled stamps with leaf 1 on top. $b(n)$: number of foldings of $n$ blank stamps. $k(p)$: number of symmetric foldings of $n=2p+1$ blank stamps.}
\label{table_stamps}
\end{table}

\begin{table}[ht]
\footnotesize
\begin{center}
\begin{tabular}{| l l l l |}
\hline
$n$ &$m(n)$ &$a(n)$ &$k^o(p)$\\
\hline
1	&1									&1					         &1\\
2	&1									&1					         &\\
3	&2									&1					         &1\\
4	&3									&2					         &\\
5	&8									&3					         &2\\
6	&14								&8					         &\\
7	&42								&13				         &5\\
8	&81								&42				         &\\
9	&262								&72				         &13\\
10	&538								&273				         &\\
11	&1828								&475				         &36\\
12	&3926								&1970				         &\\
13	&13820							&3506				         &102\\
14	&30694							&15368			         &\\
15	&110954							&27888			         &299\\
16	&252939							&126510			         &\\
17	&933458							&233809			         &889\\
18	&2172830							&1086546			         &\\
19 &8152860							&2039564			         &2698\\
20	&19304190						&9652364			         &\\
21	&73424650						&18360296		         &8267\\
22	&176343390						&88172609		         &\\
23	&678390116						&169610371		         &25684\\
24	&1649008456						&824506191		         &\\
25	&6405031050						&1601297937		         &80349\\
26	&15730575554					&7865294687					&\\
27	&61606881612					&15401847339				&253872\\
28	&152663683494					&76331857094				&\\
29	&602188541928					&150547538649				&806334\\
30	&1503962954930					&751981532942				&\\
31	&5969806669034					&1492452957398				&2580279\\
32	&15012865733351				&7506432993145				&\\
33	&59923200729046				&14980804327584			&8290645\\
34	&151622652413194				&75811326673326			&\\
35	&608188709574124				&152047190790814			&26794566\\
36	&1547365078534578				&773682540353704			&\\
37	&6234277838531806				&1558569503073541			&86881179\\
38	&15939972379349178			&7969986193751019			&\\
39	&64477712119584604			&16119428171413211		&283034120\\
40	&165597452660771610			&82798726340037900		&\\
41	&672265814872772972			&168066454180454102		&924521718\\
42	&1733609081727968492			&866804540900696571		&\\
43	&7060941974458061392			&1765235495130283117		&3031535538\\
\hline
45	&									&								&9962795554\\
\hline
\end{tabular}
\end{center}
\caption{$m(n)$: number of meanders through $n$ bridges. $a(n)$: number of meander shapes with $n$ crossings. $k^o(p)$: number of symmetric meander shapes with $n=2p+1$ crossings.}
\label{table_meanders}
\end{table}

\bigskip
\hrule
\bigskip

\noindent 2000 \textit{Mathematics Subject Classification}: Primary 05A05; Secondary 11A05.

\noindent \textit{Keywords}: stamp foldings, symmetric foldings, meanders, permutations, cyclic shift.

\bigskip
\hrule
\bigskip

\end{document}